\documentclass{amsart}
\usepackage{latexsym,amsmath,amsfonts,eucal,amscd,amssymb,mathrsfs, graphpap,mathdots,comment}

\def\rr{\mathbb{R}}
\def\cc{\mathbb{C}}
\def\zz{\mathbb{Z}}
\def\qq{\mathbb{Q}}
\def\pp{\mathbb{P}}

\def\aa{\mathbb{A}}

\def\too{\longrightarrow}

\def\ggg{\mathbb{G}}
\def\iter#1#2{{#1}^{(#2)}}
\def\itnc#1#2{{#1}_{\mathrm{nc}}^{(#2)}}
\def\itncpsi#1#2{{#1}_{\mathrm{nc},\psi}^{(#2)}}
\def\itlin#1#2{{#1}_{\mathrm{lin}}^{(#2)}}
\def\itersub#1#2#3{{#1}_{#2}^{(#3)}}

\newtheorem{theorem}{Theorem}

\newtheorem{corollary}{Corollary}
\newtheorem{conjecture}{Conjecture}
\newtheorem{proposition}{Proposition}
\theoremstyle{definition}
\newtheorem{definition}{Definition}
\newtheorem{remark}{Remark}
\newtheorem{example}{Example}

\begin{document}

\title{Integral Points and Relative Sizes of Coordinates of Orbits in $\pp^N$}
\author{Yu Yasufuku}
\address{Nihon University, College of Science and Technology, Department of Mathematics, 1-8-14 Kanda-Surugadai,  Chiyoda, Tokyo 101-8308, JAPAN}
\email{yasufuku@math.cst.nihon-u.ac.jp}

\thanks{Supported in part by JSPS Grants-in-Aid 23740033 and by Nihon University College of Science and Technology Grants-in-Aid for Fundamental Science Research.
}

\begin{abstract}
We give a generalization to higher dimensions of Silverman's result on finiteness of integer points in orbits.  Assuming Vojta's conjecture, we prove a sufficient condition for morphisms on $\pp^N$ so that $(S,D)$-integral points in each orbit are Zariski-non-dense.  This condition is geometric, and for dimension $1$ it corresponds precisely to Silverman's hypothesis that the second iterate of the map is not a polynomial.  In fact, we will prove a more precise formulation comparing local heights outside $S$ to the global height.  For hyperplanes, this amounts to comparing logarithmic sizes of the coordinates, generalizing Silverman's precise version in dimension $1$.  We also discuss a variant where we can conclude that integral points in orbits are finite, rather than just Zariski-non-dense. Further, we show unconditional results and examples, using Schmidt's subspace theorem and known cases of Lang--Vojta conjecture.  We end with some extensions to the case of rational maps and to the case when the arithmetic of the orbit under one map is controlled by the geometric properties of another.  We include many explicit examples to illustrate different behaviors of integral points in orbits in higher dimensions.
\end{abstract}

\maketitle


\section{Introduction}


Dynamics is a field studying asymptotic behavior of iterations of a self-map.  We will denote the $m$-fold iterations $\underbrace{\phi\circ\cdots \circ \phi}_m$ by $\iter \phi m$; we will not use higher derivatives so this should not cause any confusion.  Dynamics was classically studied over $\cc$, but more recently arithmetic dynamics, studying number-theoretic behaviors of self-maps defined over number fields, has now also become a very active field. For example, one would like to know how frequently ``integral points'' occur on the orbit $\mathcal O_\phi(P) = \{\iter \phi m(P):m = 0,1,2,\ldots\}$ of a rational point $P$.  Silverman \cite{silrat} answered this question for rational functions in one variable, as follows:
\begin{theorem}[Silverman]\label{thm:sil}
Let $\phi\in \qq(z)$ be a rational function of degree $\ge 2$.  \\
\emph{(i)} If $\iter \phi 2$ is not a polynomial (i.e. not in $\qq[z]$), then $\mathcal O_\phi(P) \cap \zz$ is a finite set for any $P\in \qq$.\\
\emph{(ii)} Assume that neither $\iter \phi 2$ nor $\frac 1{\iter\phi 2 (1/z)}$ is a polynomial.  If we write $\iter \phi m (P) = a_m/b_m$ in a reduced form, then for any $P\in \qq$ with $|\mathcal O_\phi(P)| = \infty$,
\[
\lim_{m\too \infty} \frac{\log |a_m|}{\log |b_m|} = 1.
\]
\end{theorem}

Viewing a rational function as a morphism on $\pp^1$, the hypothesis of being a polynomial is equivalent to having a totally ramified fixed point at $\infty$.  So the upshot is that unless $\phi$ satisfies a special ramification property, only finitely many points in any orbit are integral and in fact orbit points asymptotically become further and further away from being integral.  The proof of this theorem uses the powerful Diophantine approximation theorem of Roth, together with a combinatorial description of ramification on iterations based on Riemann--Hurwitz formula.

In this article, we will prove several results that generalize Silverman's theorem to higher-dimensions.  Dynamics in higher-dimension is in general very difficult.  Some results are known for powers of $\pp^1$: for example, the dynamical Mordell--Lang conjecture can be proved for $(\varphi,\ldots, \varphi)$ on $(\pp^1)^g$ under some conditions on $\varphi$ \cite{bgkt}, the dynamical Manin--Mumford conjecture is known for lines in $(\pp^1)^2$ \cite{GTZh}, and  integral points in orbits of $(\varphi, \varphi)$ on $(\pp^1)^2$  are known to be finite \cite{cstz}.  These may appear to be highly special at the first glance, but they are already quite difficult.  Dynamics on projective spaces in which multiple variables intermingle seems to be even more challenging.

There are also obstacles specific to generalizing Silverman's result to higher-dimensions.  Since the notion of integrality is defined with respect to a divisor, it is natural to look at the pullbacks of this divisor by iterates, just as Silverman analyzes poles of iterates.  However, here is one difficulty: irreducible divisors in $\pp^1$ are simply points, but irreducible divisors in higher-dimensions can be highly singular and can even have self-intersections.  Moreover, while there has been great recent progress in Diophantine approximation such as \cite{ever_fer}, unfortunately they are not strong enough for our purposes.  Thus, we will resort to a very deep Diophantine conjecture of Vojta \cite[Conjecture 3.4.3]{vojta} to treat general situations, while giving many illuminating examples where the use of this conjecture can be avoided.  We note that a stronger version of Vojta's conjecture has been applied to arithmetic dynamics in \cite{sil_zsig} to analyze dynamical Zsigmondy sets for $\pp^N$.

We now discuss our main results.  We will always use the convention that the \textit{degree} of a rational map $\pp^N\dashrightarrow \pp^N$ is the polarization degree.  Given an effective divisor $D$ on $\pp^N$ defined over a number field, we consider all the normal-crossings subdivisors of $D$ over $\overline \qq$ and the one with the highest degree will be called \textit{a normal-crossings part} of $D$, denoted by $D_{\mathrm{nc}}$.  We denote the pullback $(\iter \phi m)^*D$ by $\iter Dm$ and its normal-crossings part by $\itnc Dm$ when the map $\phi$ is clear.  We are now ready to state the main results.

\begin{theorem}[cf. Theorem \ref{thm:main}]\label{thm:sec1_fin}
Let $\phi: \pp^N \too \pp^N$ be a morphism defined over $\qq$ of degree $d\ge 2$, and let $D$ be a divisor on $\pp^N$.  If $\deg \itnc Dn > N+1$ for some $n$, then Vojta's conjecture on $\pp^N$ for the divisor $\itnc Dn$ implies that for any $P\in \pp^N(\qq)$, $\mathcal O_\phi(P) \cap (\pp^N\backslash D)(\zz)$ is Zariski-non-dense.
\end{theorem}

If $D$ is the hyperplane $X_N = 0$, then by writing $\iter \phi m (P) = [\itersub a0m : \cdots :\itersub aNm]$ with $\itersub aim\in \zz$ with common divisor $1$, $\iter \phi m(P)$ is in $(\pp^N\backslash D)(\zz)$ if and only if $\itersub aNm = \pm 1$.  So the theorem says that the Zariski-closure of integer-coordinate points in orbits is not all of $\pp^N$.  Note that Theorem \ref{thm:sec1_fin} for $N=1$ is exactly Theorem \ref{thm:sil} (i).  Indeed, using Riemann--Hurwitz,  Silverman proves that $\iter \phi 2$ is not a polynomial if and only if some iterate has at least three distinct poles. Since Zariski-non-denseness is equivalent to finiteness on $\pp^1$, Theorem \ref{thm:sec1_fin} for $N=1$ completely agrees with Silverman's result.  
This is an upgrade from a prior work \cite{vojta_dyn}, as the hypothesis in the main theorem there was strictly stronger for $N=1$ than Silverman's theorem.

Just as in dimension $1$, we can also obtain a more precise version involving the number of digits of the coordinates, albeit assuming Vojta's conjecture for more divisors:

\begin{theorem}[cf. Corollary \ref{cor:coords}]\label{thm:sec1_size}
Let $D$ be the hyperplane $X_N = 0$, and let $c = \sup_n \frac {\deg \itnc Dn - (N+1)}{d^n}$.  Assuming Vojta's conjecture on $\pp^N$, for any $P\in \pp^N(\qq)$ and $\epsilon>0$, $\displaystyle \left\{\iter \phi m(P): \frac{\log \left| \itersub aNm\right|}{\log \max_i \left| \itersub aim\right|} \le c-\epsilon\right\}$
is Zariski-non-dense.
\end{theorem}

To make up for having to use a deep conjecture to treat general maps, we include many explicit examples throughout the paper which do not require Vojta's conjecture.  In fact, the examples are important components of this paper.  In particular, Examples \ref{ex:schmidt}, \ref{ex:vojtasemi}, and \ref{ex:archS} provide different circumstances where assuming this conjecture can be circumvented, and Examples \ref{ex:ratl_infty_int}, \ref{ex:ratl_wo_dratio}, and \ref{ex:dratio} analyze interesting rational maps without assuming any conjecture. They should shed some light to the spectrum of behaviors for integral points in orbits in higher-dimensions.

We now describe the results and examples in this paper in slightly more detail.  In Section 2, we recall definitions of global and local heights and of normal-crossings divisors, and then discuss Vojta's conjecture. In Section 3, we will prove the number-field versions of Theorems \ref{thm:sec1_fin} and \ref{thm:sec1_size}, involving several places.  These theorems and their several variants are all consequences of Theorem \ref{thm:main}.  One of the variants (Corollary \ref{cor:generic}) allows us to conclude ``finiteness'' rather than ``Zariski-non-denseness'' under the assumption that the orbit of $P$ is \textit{generic}, that is, any infinite subset of $\mathcal O_\phi(P)$ is Zariski-dense.  It is one of the far-reaching and difficult conjectures of Zhang \cite{zhang} that the orbit of most points is generic, but we will show some examples where finiteness can be concluded unconditionally.

In Section 4, we discuss results and examples for which Vojta's conjecture is not necessary.  For example, if $\itnc Dn$ is a union of hyperplanes, then Schmidt's subspace theorem gives us the same results unconditionally (Proposition \ref{prop:schmidt}).  This situation has a bonus that if $\itnc Dn$ is defined over $\qq$, then the exceptions to the subspace theorem are also hyperplanes defined over $\qq$.  This observation is exploited in Example \ref{ex:schmidt}, and we discuss its connection with dynamical (rank-one) Mordell--Lang conjecture.  Another situation where the results become unconditional comes from a weaker form of Vojta's conjecture called Lang--Vojta conjecture for integral points.  We take advantage of known cases of this conjecture in Proposition \ref{prop:vojtasemi} and Example \ref{ex:vojtasemi}.  We also mention an example (Example \ref{ex:bad}) which demonstrates that our criterion given in Theorem \ref{thm:sec1_fin} for Zariski-non-denseness of integral points in orbits is not yet satisfactory.

In Section 5, we discuss some extensions.  The first extension is removing the assumption that $\phi$ is a \textit{morphism}.  For this, we will use the notion of $D$-ratio, introduced by Lee \cite{joey_aa}.  Using his height inequality for rational maps, we prove an extension to rational maps (Theorem \ref{thm:dratio}).  We also give an explicit example of this theorem for which Vojta's conjecture is unnecessary (Example \ref{ex:dratio}).  The second extension is a case when the arithmetic of the orbit under one map is controlled by the geometry of another (Theorem \ref{thm:twomaps}).  We mention several open questions at the very end.

\section{Background on Vojta's Conjecture}

In this section, we briefly introduce Vojta's conjecture \cite{vojta}. It is an extremely deep inequality of heights, and its consequences are vast, including Mordell's conjecture (Faltings \cite{falcurves}), Schmidt's subspace theorem, Bombieri--Lang conjecture and the $abc$ conjecture (Vojta \cite{vojtaabc}).  There are multiple versions of the conjecture, including the so-called Lang--Vojta conjecture \cite[Proposition 4.1.2]{vojta} which specializes to integral points
and a uniform one over $\overline \qq$ \cite[Conjecture 5.2.6]{vojta}, but the version we will use in this article is an inequality for rational points over a fixed number field \cite[Conjecture 3.4.3]{vojta}.

We first recall important basic properties of heights, and set some notations.  Let $k$ be a number field, and let $M_k$ be the set of absolute values up to equivalence. For each $v\in M_k$, let
$|\cdot|_v$ be the absolute value in the class of $v$ which is the $\frac{[k_v:\qq_v]}{[k:\qq]}$-th power of the extension of the
normalized absolute value on $\qq$, so that the product formula is simply $\prod_v  |x|_v = 1$ for $x\in k^*$. We sometimes use the additive notation  $v(x) = -\log |x|_v$.  When $S$ is a finite subset of $M_k$ containing all of the archimedean ones, the ring $R_S$ of \textit{$S$-integers} is the set of all $x\in k$ such that $|x|_v \le 1$ for all $v\notin S$.  We define \textit{Weil height} on $\pp^N(\overline \qq)$ by
\[
h([x_0:\cdots :x_N]) = \sum_{v\in M_k}  \log \max (|x_0|_v,\ldots,|x_N|_v),
\]
which is well-defined.  When $\phi: \pp^N \dashrightarrow \pp^M$ is a rational (algebraic) map, we define the \textit{degree} of $\phi$ to be the common degree of the homogeneous polynomials defining coordinates of $\phi$.  This can also be viewed as the degree of the map on the Picard group.  We have the obvious height inequality $h(\phi(P)) \le d h(P) + O(1)$, and whenever $\phi$ is a morphism (i.e. it is defined everywhere without indeterminancy), we also have an important inequality in the opposite direction:
\begin{equation}\label{lowerht}
h(\phi(P)) > d h(P) + O(1).
\end{equation}
Let $X$ be a projective variety over $k$, assumed to be irreducible unless otherwise stated.  For any Cartier divisor $D$ on $X$, we can define a Weil height $h_D$ by writing $D$ as a difference of ample divisors and using the heights on the projective spaces.  We can also define \textit{local height} $\lambda_v(D,-):X(k)\backslash |D| \too \rr$ for $v\in M_k$ and a divisor $D$, using an $M_k$-bounded metric on the line bundle $\mathscr L(D)$. In essence, $\lambda_v(D,P)$ is  $-\log |f(P)|_v$, where $f$ is a local equation for $D$, but one needs to glue this together in a coherent way using $M_k$-bounded functions.  As $\lambda_v(D,P)$ is big when $P$ is $v$-adically close to $D$, this is the number-theoretic analog of the proximity function in Nevanlinna theory. For details, see \cite{bomgub}, \cite{langdg}.  With our normalization, we have
\begin{equation}\label{eq:sum_local}
\sum_{v\in M_k} \lambda_v(D, P) = h_D(P) + O(1), \qquad \forall P\in X(k)\backslash |D|.
\end{equation}
Local height functions also satisfy functoriality with respect to pullbacks: given $\phi:Y\too X$,
\begin{equation}\label{eq:functorial}
\lambda_v(D, \phi(Q)) = \lambda_v(\phi^*(D), Q) + O(1), \qquad \forall Q\in Y(k)\backslash \phi^{-1}(|D|),
\end{equation}
and in fact the inequality with $\le$ holds even for rational maps $\phi:Y\dashrightarrow X$.

On $\pp^N$, if a  divisor $D$ is defined (globally) by the homogeneous polynomial $F$ of degree $d$, then a local height is simply
\begin{equation}\label{eq:localht}
\lambda_v(D,[x_0:\cdots : x_N]) = v(F(x_0,\ldots,x_N)) - d \min(v(x_0),\ldots, v(x_N)).
\end{equation}
In particular, choosing $F$ to have coefficients in the ring of integers, we see that $\lambda_v(D,P)\ge 0$ for any non-archimedean $v$.  We also note that $h_D(P) = d h(P) + O(1)$, but a similar relation does not hold with a single $\lambda_v(D,-)$.

For $k=\qq$, let us write $P\in \pp^N(\qq)$ as $[a_0 : \cdots :a_N]$, where $a_i\in \zz$ with common divisor $1$.  Then letting $H = (X_N=0)$,
\begin{equation}\label{eq:localhtQ}
\lambda_{v_p}(H, P) = \log \max_i |a_i|_{v_p} - \log |a_N|_{v_p} = \log  (p\text{-part of } |a_N|)
\end{equation}
for each place $v_p$ corresponding to the prime $p$. If $S \subset M_\qq$ a finite subset including the archimedean place $v_\infty$, then we denote the \textit{prime-to-$S$ part} of an integer $x$ by $|x|_S'$.  By above,
\begin{equation}\label{eq:primetoS}
\log |a_N|_S' = \sum_{v\notin S} \lambda_v(H, P).
\end{equation}
Thus, $\{P\in (\pp^N\backslash H)(\qq): \sum_{v\notin S} \lambda_v(H,P) = 0\}$ is precisely the set $(\pp^N\backslash H)(R_S)$ of points with $S$-integer coordinates, i.e. $[a_0:\cdots : a_{N-1}:1]$ with $a_i\in R_S$.  In general, we say a set is \textit{$(S,D)$-integral} if it is of the form
\[
\{P\in (X\backslash D)(k): \lambda_v(D,P) \le c_v \text{ for } v\notin S\},
\]
where almost all $c_v$ are $0$.  We will abuse the notation and write $(X\backslash D)(R_S)$ for an $(S,D)$-integral set.

We say that a divisor on a smooth variety is \textit{normal-crossings} if near each point the divisor is defined by $x_1\cdots x_k = 0$, where $x_1,\ldots, x_k$ is a part of a local (analytic) coordinate system.  Note that by definition, the multiplicity of each irreducible component in a normal-crossings divisor is $1$.   We are now ready to state Vojta's conjecture \cite[Conjecture 3.4.3]{vojta}.

\begin{conjecture}[Vojta's Conjecture]\label{con:vojta}
Let $X$ be a smooth projective variety over $k$, $K$ a canonical divisor of $X$, $A$ an ample divisor and $D$ a normal-crossings divisor. Fix height functions $\lambda_v(D, -)$, $h_K$, and $h_A$.  Let $S\subset M_k$ be a finite set of places.  Then given $\epsilon>0$, there exists a Zariski-closed $Z_\epsilon\subsetneq X$ such that
\begin{equation}\label{eq:vojtacon}
\sum_{v\in S} \lambda_v(D, P) + h_K(P) < \epsilon h_A(P) +O(1)
\end{equation}
for all $P\in X(k)$ not on $Z_\epsilon$.
\end{conjecture}

Since local heights have logarithmic poles along $D$, this conjecture states that a point cannot get too close $v$-adically to $D$ for $v\in S$, and how close a rational point can approximate $D$ is controlled by the geometry of the variety, namely how negative the canonical divisor is.  We note that the normal-crossings assumption on the divisor is absolutely essential, and this condition will be important in the rest of the paper.  Since we mostly work with projective spaces, we also state the following version.

\begin{conjecture}[Vojta's Conjecture for $\pp^N$]
Let $D$ be a normal-crossings divisor on $\pp^N$ defined over $k$, and $S$ be a finite set of places.  Then given $\epsilon>0$, there
exists a finite union $Z_\epsilon$ of hypersurfaces  and a constant $C$ such that for any $P\in \pp^N(k)\backslash Z_\epsilon$,
\[
\sum_{v \in S} \lambda_v(D, P) < (N+1+\epsilon) h(P) + C.
\]
\end{conjecture}

This is precisely  Roth's theorem when $N=1$ and $S$ consists of a single archimedean absolute value.  In fact, if $D$ is a union of hyperplanes in $\pp^N$ in general position, this conjecture can be shown to be equivalent to Schmidt's subspace theorem.  Thus, one can view Vojta's conjecture as a higher-degree extension of the subspace theorem.



\section{Proofs of the Theorems}

We will first prove the following technical theorem, from which Theorems \ref{thm:sec1_fin} and \ref{thm:sec1_size} and other variants can be easily derived.

\begin{theorem}\label{thm:main}
Let $\phi: \pp^N \too \pp^N$ be a morphism defined over $\overline \qq$ of degree $d\ge 2$.  Let $D$ be a divisor on $\pp^N$ defined over $\overline \qq$, and let $\itnc Dn$ be the normal-crossings part of $(\iter \phi n)^*(D)$.    Let $c_n = \frac{\deg \itnc Dn - (N+1)}{\deg(D)\cdot d^n}$, and let $k$ be a number field that contains the fields of definition of $\phi$, $D$, and $\itnc Dn$.  Assume Vojta's conjecture for the divisor $\itnc Dn$.  Then for any $P\in \pp^N(k)$, for any finite set $S\subset M_k$, and for any $\epsilon >0$, the set
\[
\left\{\iter \phi m(P): \,\,\frac{\displaystyle \sum_{v\notin S} \lambda_v(D, \iter \phi m(P))}{\deg(D)\cdot h(\iter \phi m(P))} \le c_n -\epsilon\right\}
\]
is Zariski-non-dense.  In particular, $\mathcal O_\phi(P) \cap (\pp^N\backslash D)(R_S)$ is Zariski-non-dense if $c_n>0$ for some $n$.
\end{theorem}

\begin{proof}
By applying Vojta's conjecture to the divisor $\itnc Dn$ on $\pp^N$, there exists a constant $C$ such that
\begin{equation}\label{eq1}
\sum_{v\in S} \lambda_v \left(\itnc Dn, Q\right) < (N+1+ \epsilon) h(Q) + C
\end{equation}
holds for all $Q\in \pp^N(k)$ except for points on a finite union $Z_\epsilon$ of hypersurfaces.  Since $\phi$ is a morphism, the degree of $\iter Dn = (\iter \phi n)^*D$ is $d^n \deg(D)$, and thus we also have
\begin{equation}\label{eq2}
\sum_{v\in S} \lambda_v \left(\iter Dn - \itnc D n, Q\right) \le h_{\iter Dn - \itnc D n}(Q) +C_1' \le \left(d^n \deg(D)- \deg \itnc Dn\right) h(Q) + C_1.
\end{equation}
Then for $Q\notin Z_\epsilon$,
\begin{align*}
\sum_{v\in S} &\lambda_v(D, \iter \phi n(Q)) = \sum_{v\in S} \lambda_v((\iter \phi n)^*(D), Q) + C_2 &&\text{functoriality \eqref{eq:functorial}}\\
&<\left( d^n \deg(D) - \deg \itnc Dn + (N+1) + \epsilon \right) h(Q) + C_3 &&\text{\eqref{eq1} and \eqref{eq2}}\\
&\le \left(1-c_n + \frac \epsilon {d^n\deg(D)}\right) \deg(D) h(\iter \phi n (Q)) + C_4 &&\phi \text{ a morphism and \eqref{lowerht}}\\
&=\left(1-c_n + \frac \epsilon {d^n\deg(D)}\right) h_D(\iter \phi n (Q)) + C_5.
\end{align*}
Hence, for any $m\ge n$, we let $Q = \iter \phi {m-n}(P)$ and conclude that
\begin{align*}
\sum_{v\notin S} \lambda_v(D, \iter \phi m(P)) &> \left(c_n - \frac \epsilon{d^n \deg(D)}\right) h_D(\iter \phi m(P)) - C_6\\
&> \left(c_n - \frac \epsilon{d^n \deg(D)}\right) \deg(D) h(\iter \phi m(P)) - C_7
\end{align*}
as long as $\iter \phi {m-n}(P) \notin Z_\epsilon$.  Note that if $\mathcal O_\phi(P)$ is finite, then this theorem is trivial.  Otherwise, $h(\iter \phi m(P))\to \infty$ as $m\to \infty$ by Northcott's theorem. By dividing both sides of the inequality by $\deg(D) h(\iter \phi m (P))$, $C_7$ can be incorporated into a change in $\epsilon$ for large enough $m$'s.  Moreover, if $\iter \phi {m-n}(P) \in Z_\epsilon$, then $\iter \phi m(P) \in \iter \phi n(Z_\epsilon)$, which is a Zariski-closed set not equal to the whole of $\pp^N$.  Therefore, the result follows, as the given set is contained in the union of $\iter \phi n(Z_\epsilon)$ with a finitely many of the orbit points.  The last sentence of the theorem is immediate from the discussion of $(S,D)$-integral sets in the previous section.
\end{proof}

\begin{remark}
What we actually prove is the following: there exists a constant $C$ such that $\left\{\iter \phi m(P): \sum_{v\notin S} \lambda_v(D, \iter \phi m(P)) \le (c_n -\epsilon) \deg(D)\cdot h(\iter \phi m(P)) - C \right\}$ is Zariski-non-dense.  The constant $C$ comes partially from \eqref{eq:vojtacon}, so it is not effective.  In the following, we will prove similar results in various settings, all of which can be stated as differences of heights, although we state them with ratios for simplicity.
\end{remark}

\begin{corollary}\label{cor:sup}
Let $c_n = \frac{ \deg \itnc Dn - (N+1)}{d^n \deg(D)}$ and $c =  \sup_n c_n$.  Then assuming Vojta's conjecture for $\pp^N$, for all $\epsilon >0$,
\begin{equation}\label{eq:coord_set}
\left\{\iter \phi m(P): \,\, \frac{\displaystyle \sum_{v\notin S} \lambda_v(D,\iter\phi m(P))}{\deg(D) h(\iter\phi m (P))} \le c-\epsilon \right\}
\end{equation}
is Zariski-non-dense.
\end{corollary}

\begin{proof}
This is vacuous if $c\leq 0$ (from \eqref{eq:localht}, each $\lambda_v(D,-)$ is nonnegative), and if not, there exists $n$ such that $c-\epsilon < c_n \leq c$ with $c_n >0$.  Then using $\epsilon' = \frac{c_n - (c-\epsilon)}2$ in Theorem \ref{thm:main} shows the result, as $c_n - \epsilon' > c-\epsilon$.
\end{proof}

\begin{remark}
In truth, if $\{c_n: n\in I\}$ is a monotone increasing sequence whose limit is $c$, then we only need to assume Vojta's conjecture for divisors $\itnc Dn$ with $n \in I$.
\end{remark}

Note that $c$ is completely determined geometrically, and it does not depend on the choice of $S$.  On the other hand, the (possibly reduced) proper subvariety that contains \eqref{eq:coord_set} will depend on $\epsilon$ and $S$ (conjecturally, the hypersurface part does not depend on $S$, but the additional higher-codimensional part will certainly depend on $S$).

We next discuss another variation of the main results.  An infinite set of rational points is called \textit{generic} if any infinite subset is Zariski-dense.  Zhang \cite{zhang} has conjectured  that any polarized dynamical system has a point whose orbit is generic.  If we assume genericity of the orbit, then we can conclude \textit{finiteness}, rather than just Zariski-non-denseness, as follows.

\begin{corollary}\label{cor:generic}
Let $c = \sup_n \frac{ \deg \itnc Dn - (N+1)}{d^n \deg(D)}$.  Let us assume that the orbit of $P$ is generic (in particular, $|\mathcal O_\phi(P)|$ is infinite). Assuming Vojta's conjecture for $\pp^N$, for all $\epsilon >0$, then
\begin{equation*}
\frac{\displaystyle \sum_{v\notin S} \lambda_v(D,\iter\phi m(P))}{\deg(D) h(\iter\phi m (P))} > c-\epsilon
\end{equation*}
holds for sufficiently large $m$.  In particular,\\
 \emph{(i)} If there exists $n$ such that  $\deg \itnc Dn > N+1$, then $\mathcal O_\phi(P) \cap (\pp^N\backslash D)(R_S)$ is a finite set.\\
\emph{(ii)} If $c = 1$, then
\[
\lim_{m\to \infty} \frac{\displaystyle \sum_{v\notin S} \lambda_v(D,\iter\phi m(P))}{\deg(D) h(\iter\phi m (P))}
\]
exists and equal to $1$.
\end{corollary}

\begin{proof}
By the definition of genericity, the only Zariski-non-dense subset of the orbits is a finite set, so the first statement follows immediately from Corollary \ref{cor:sup}.  (i) then follows as before from the fact that $h(\iter \phi m(P)) \to \infty$,  and (ii) follows from the fact that the numerator inside the limit is $\le h_D(\iter \phi m(P)) = \deg(D) h(\iter \phi m(P)) + O(1)$, together with the squeeze theorem.
\end{proof}

We now specialize Theorem \ref{thm:main} and its corollaries to $k= \qq$ and $D = H = (X_N = 0)$, and derive Theorem \ref{thm:sec1_size}.  For $P\in \pp^N(\qq)$, let us write $\iter \phi m (P) = [\itersub a0m : \cdots :\itersub aNm]$, where $\itersub aim\in \zz$ with common divisor $1$.

\begin{corollary}\label{cor:coords}
Let $c = \sup_n \frac{ \deg \itnc Hn - (N+1)}{d^n}$, and let $S \subset M_\qq$ be a finite subset containing $v_\infty$.  \\
\emph{(i)} If there exists $n$ such that $\deg \itnc Hn > N+1$, then assuming Vojta's conjecture for $\itnc Hn$, for any $P\in \pp^N(\qq)$, the $(S,H)$-integral points $\{\iter \phi m(P): |\itersub aNm|_S' = \pm 1\}$ in the orbit of $P$ is Zariski-non-dense.  \\
\emph{(ii)} Assuming Vojta's conjecture for $\pp^N$, for all $\epsilon >0$,
\[
\left\{\iter \phi m(P): \frac{\log |\itersub aNm|_S'}{\log \max_i |\itersub aim|}  \le c-\epsilon \right\}
\]
is Zariski-non-dense. \\
\emph{(iii)} If the orbit of $P$ is generic, then assuming Vojta's conjecture for $\pp^N$, for all $\epsilon >0$,
\[
\frac{\log |\itersub aNm|_S'}{\log \max_i |\itersub aim|} > c-\epsilon
\]
holds for all sufficiently large $m$.  In particular, if there exists $n$ such that $\deg \itnc Hn > N+1$, $\{\iter \phi m(P): |\itersub aNm|_S' = \pm 1\}$ is a finite set, and if $c = 1$,
\[
\lim_{m\to \infty} \frac{\log |\itersub aNm|_S'}{\log \max_i |\itersub aim|} = 1.
\]
\end{corollary}

\begin{remark}
If $S = \{v_\infty\}$, $|x|_S' = |x|$ for any integer, so we obtain Theorem \ref{thm:sec1_size}.  Moreover, in the case of $\pp^1$, if the orbit is infinite (i.e. $P$ is not preperiodic), then it is automatically generic.  Therefore, the last part of (iii) generalizes Silverman's coordinate-size result (Theorem \ref{thm:sil} (ii)).
\end{remark}

\begin{proof}
These all follow directly from Theorem \ref{thm:main} and Corollaries \ref{cor:sup} and \ref{cor:generic}, using \eqref{eq:localhtQ} and \eqref{eq:primetoS}.
\end{proof}

\section{Unconditional Results and Examples}

In this section, we discuss some cases for which we obtain results such as Theorem \ref{thm:main} unconditionally without assuming any conjecture.  One major case comes from Schmidt's subspace theorem (Proposition \ref{prop:schmidt} and Example \ref{ex:schmidt}).  There are other special cases for which Vojta's conjecture can be proved, and we discuss these examples as well (Proposition \ref{prop:vojtasemi} and Examples \ref{ex:vojtasemi} and \ref{ex:archS}).

First, we discuss cases for which Schmidt's subspace theorem applies.  Since this theorem is equivalent to Vojta's conjecture for $X = \pp^N$ and $D$ a union of hyperplanes in general position, whenever normal-crossings divisors are linear, we get results in the previous section unconditionally.  We now make this precise.  Similar to the definition of $\itnc Dn$, we define a \textit{linear normal-crossings part} $\itlin Dn$ of $\iter Dn$ to be a highest-degree normal-crossings subdivisor over $\overline \qq$ of $\iter Dn$ whose support is a union of hyperplanes.  In other words, among all the different general-position unions of $\overline \qq$-hyperplanes contained in $|\iter Dn|$, $\itlin Dn$ has the most number of components.

\begin{proposition}\label{prop:schmidt}
Let $\phi: \pp^N\too \pp^N$ be a morphism defined over $\overline \qq$, let $D$ be an effective divisor on $\pp^N$ defined over $\overline \qq$, and $P\in \pp^N(\overline \qq)$.  Suppose there exists $n$ such that  $c_n' = \frac{ \deg \itlin Dn - (N+1)}{d^n \deg(D)}$ is positive. Let $k$ contain the fields of definition of $\phi$, $D$, $\itlin Dn$, and $P$, and let $S$ be a finite subset of $M_k$.   Then for all $\epsilon >0$,
\begin{equation}\label{eq:linearsup}
\left\{\iter \phi m(P): \frac{\displaystyle \sum_{v\notin S} \lambda_v(D,\iter \phi m(P))}{\deg(D) h(\iter \phi m(P))} \le c_n'-\epsilon \right\}
\end{equation}
is Zariski-non-dense.  Further, if no hyperplane defined over $k$ contains infinitely many  points of $\mathcal O_\phi(P)$, then \eqref{eq:linearsup} is a finite set.
\end{proposition}

As before, we also have a version for $(S,D)$-integral points, a version with $c' = \sup_n c_n'$, and a version with $|\cdot |_S'$.

\begin{remark}
Much progress has been made on Schmidt's subspace theorem and its exceptional hyperplanes.  For example, we have some upper bounds for the number of exceptional hyperplanes. However, in general we still do not have a bound for the heights of the coefficients of the exceptional linear subspaces, and so we do not have an effective bound of $m$ which lies in \eqref{eq:linearsup}.
\end{remark}

\begin{proof}
All of these follow directly from the corresponding statements involving $\itnc Dn$ instead of $\itlin Dn$.  One notable observation is the fact that Schmidt's subspace theorem actually lets us conclude that the exceptions to the inequality of Vojta's conjecture are contained in a finite union of hyperplanes \textit{defined over} $k$. For general normal-crossings divisors, we only know that the exception is a union of hypersurfaces, so this is much stronger. To conclude finiteness of \eqref{eq:linearsup}, we only need to show that each hyperplane over $k$ contains only finitely many points of $\mathcal O_\phi(P)$.
\end{proof}

\begin{example}\label{ex:schmidt}
Let $\phi$ be the morphism $[Y^4 + Z^4: X^3 (X+Y+Z) : YZ^3]$ on $\pp^2$. Letting $D = (Z =0)$, $\iter D2 = (X^3 (X+Y+Z)Y^3 Z^9 = 0)$, and $\itlin D2 = \itnc D2 = (XYZ(X+Y+Z)=0)$, satisfying the hypothesis of the proposition.  Hence, $(S,D)$-integral points in the orbit are Zariski-non-dense, and more precisely,
\begin{equation}\label{eq:line_ex}
\left\{\iter\phi m (P) = [\iter am: \iter bm : \iter cm]:  \frac{\log |\iter cm|_S'}{\log \max (|\iter am|, |\iter bm|, |\iter cm|)} \le \frac 1{16} - \epsilon \right\}
\end{equation}
is Zariski-non-dense.

Now, let $P = [100:2000:1]$, and we will show that no line defined over $\qq$ contains infinitely many points of $\mathcal O_\phi(P)$.  Unlike the ratio of the \textit{number of digits} of coordinates, the ratio of the coordinates is unaffected even when there is a common factor.  Since the $Z$-coordinate of $P$ is much smaller than the first two coordinates and the first two coordinates of $\phi$ are quartic in $X$ and $Y$ while the last coordinate is only linear, the ratio of the last coordinate of $\iter \phi m(P)$ to either of the first two coordinates becomes smaller and smaller in absolute value as $m\to \infty$.  Therefore, a hyperplane with a nonzero coefficient for $Z$ cannot contain infinitely many points of $\mathcal O_\phi(P)$.  The orbit points of $P$ clearly will not lie on $X=0$ or $Y=0$, so we are left to to check $X - \alpha Y = 0$ for $\alpha\in \qq^*$. Any orbit point lying on this line has its previous iterate a rational point on $Y^4 + Z^4 - \alpha(X^4 + X^3 Y + X^3 Z) = 0$. When this is a smooth curve, it has genus $3$, so this  immediately gives finiteness.  Using the Jacobian criterion, the derivatives with respect to $Y$ and $Z$ give $4 Y^3 - \alpha X^3 = 4 Z^3 - \alpha X^3 =0$.  So $X\neq 0$ and letting $\eta_1, \eta_2$ be some cube root of the rational number $\frac \alpha 4$, we have $Y = \eta_1 X$ and $Z = \eta_2 X$.  Then the derivative with respect to $X$ yields $4 X^3 + 3 X^2 Y + 3 X^2 Z = X^3 (4 + 3 \eta_1 + 3 \eta_2) = 0$.  There are two possibilities.  When $\eta_1 = \eta_2$, then $\eta_i<0$, so $\alpha<0$.  This line does not contain any orbit points of $P$, as their coordinates are all positive.  When $\eta_1 = \overline \eta_2$, then $\mathrm{Re}(\eta_i) = -\frac 23$ so $\alpha = 4 \eta_i^3 = \frac{256}{27}$.  For $m$ even, the $Y$-coordinate of $\iter \phi m(P)$ is bigger than the $X$-coordinate, so it will not be on this line.  The $X$-coordinate of $\phi(P)$ is about $7600$ times as big as the $Y$-coordinate, and for points with a much larger $X$-coordinate than the $Y$-coordinate, the first coordinate of $\iter \phi 2$ is dominated by $X^{16}$ while $X^4 Y^{12}$ dominates the second coordinate.  Hence, the ratio of the $X$-coordinate to the $Y$-coordinate becomes bigger and bigger upon every $\iter \phi 2$.  Again, we observe that a common factor will not affect the ratio of the coordinates.  Therefore, the first two coordinates of points in $\mathcal O_\phi(P)$ will never have a ratio of $\frac{256}{27}$.

Therefore, we conclude that \eqref{eq:line_ex} is a finite set for $P = [100:2000: 1]$, unconditionally without assuming any conjectures.  Note that to conclude finiteness, it was very useful to know that the exceptions are lines (so that we can make arguments involving ratios of coordinates) and that the coefficients are in $\qq$.
\end{example}

The argument above only utilizes standard methods, but it is somewhat adhoc and it is difficult to generalize to arbitrary $\phi$ of similar shape and arbitrary $P$.  On the other hand, given a specific situation, one can usually come up with a similar argument to show that a line defined over $\qq$ does not contain infinitely many of its orbit points.

\begin{example}\label{ex:schmidt_high}
Example \ref{ex:schmidt} can be generalized to higher-dimensions.  For example, let $\phi = [\prod^d L_i : X_0^d : \cdots : X_{N-2}^d: X_{N-1} X_N^{d-1}]$ on $\pp^N$, where $L_1,\ldots, L_d$ are linear forms in general position such that $\phi$ is a morphism.  Then $(\iter \phi {N+1})^{-1}(X_N=0)$ is the vanishing locus of $(\prod_i L_i) X_0\cdots X_N$, so Proposition \ref{prop:schmidt} applies.  On the other hand, it is more difficult in general to conclude that \eqref{eq:linearsup} is a finite set, as exceptions are no longer 1-dimensional.
\end{example}

\begin{remark}
As is evident in Example \ref{ex:schmidt}, even in cases for which Schmidt's subspace theorem applies, it would be beneficial to have an affirmative answer to dynamical Mordell--Lang question in order to conclude finiteness rather than Zariski-non-denseness.   This question asks whether an infinite intersection of the orbit with a subvariety forces the subvariety to be preperiodic, and it has been proved affirmatively in various settings using Skolem--Mahler--Lech method (see for example \cite{bgkt}).  The higher-rank case, involving several maps which commute with each other, was introduced in \cite{gtz}. They prove an affirmative answer in low-dimensional cases and they also demonstrated a couple of counterexamples.  The higher-rank case has now been proved to completely fail in general \cite{sca_yas}, though no counterexample has yet been found for orbits of a single map.  The case relevant to Example \ref{ex:schmidt}, namely the case of self-maps on $\pp^2$ with respect to a line, is not known.
\end{remark}

As a next situation when the results become unconditional, we take advantage of the known cases of  the ``integral point'' version of Vojta's conjecture.  This version is sometimes called Lang--Vojta conjecture, and it is a special case \cite[Propsoition 4.1.2]{vojta} of Conjecture \ref{con:vojta} in Section 2.  This conjectures that when $X\backslash D$ is of log general type, the $(S,D)$-integral points are Zariski-non-dense.  We will now use a known case of this to obtain an unconditional result:

\begin{proposition}\label{prop:vojtasemi}
Let $\phi: \pp^N \too \pp^N$ be a morphism defined over $\overline \qq$ of degree $d\ge 2$.  Let $D$ be a divisor on $\pp^N$ defined over $\overline \qq$, and let $P\in \pp^N(\overline \qq)$.  Let $k$ be a number field that contains the fields of definition of $\phi$, of irreducible components over $\overline \qq$ of $D$, and of $P$, and let $S\subset M_k$ be a finite subset. Suppose there exists $n$ with $(\iter \phi n)^*D = D_1 + D_2$ such that

$\bullet$ $D_1$ contains $N+2$ distinct geometrically-irreducible components,

$\bullet$ $\exists \alpha < \deg (D_2)$ with $\displaystyle \sum_{v\in S} \lambda_v(D_2, \iter \phi m (P)) \le \alpha h(\iter \phi m (P))+O(1) \,\,\forall m\gg 0$.\\
Then $\displaystyle \left\{\iter \phi m(P): \frac{\displaystyle \sum_{v\notin S} \lambda_v(D, \iter \phi m(P))}{(\deg D)h(\iter \phi m(P))} \le \frac{\deg(D_2) - \alpha}{(\deg D) d^n} \right\}$ is Zariski-non-dense.
\end{proposition}


\begin{proof}
The proof is similar in spirit to the proof of Theorem \ref{thm:main}.  Let $m\gg 0$ be such that $\iter \phi m (P)$ is in the set in question and let $Q = \iter \phi {m-n}(P)$.  Note that
\[
\sum_{v\notin S} \lambda_v(D_2, Q) = (\deg D_2) h(Q) - \sum_{v\in S} \lambda_v(D_2,Q) - C_1 \ge ((\deg D_2) - \alpha)h(Q) - C_1'.
\]
On the other hand,
\begin{align*}
\sum_{v\notin S} &\lambda_v(D_1,Q) + \sum_{v\notin S} \lambda_v(D_2,Q) = \sum_{v\notin S} \lambda_v((\iter \phi n)^*D, Q) = \sum_{v\notin S} \lambda_v(D, \iter \phi m P) + C_2\\
&\le \frac{\deg(D_2) - \alpha}{(\deg D) d^n} \cdot (\deg D)h(\iter \phi m(P)) +C_2  \le (\deg (D_2) - \alpha) h(Q) + C_3.
\end{align*}
Putting these two together, we see that $\sum_{v\notin S} \lambda_v(D_1, Q)$ is bounded by $C_1'+C_3$. For a fixed number field, \eqref{eq:localht} shows that there are only finitely many non-archimedean places for which the minimum positive value of $\lambda_v(D_1,-)$ is below $C_1' +C_3$.  Therefore, so $Q$ must belong to a set of $(S,D_1)$-integral points.  But Lang--Vojta conjecture is known for $\pp^N$ and a divisor with $N+2$ distinct geometrically-irreducible components (originally \cite[Corollary 2.4.3]{vojta} and a special case of \cite[Corollary 0.3]{vojta_semiabel}), so the conclusion follows by taking image by $\iter \phi n$.
\end{proof}

\begin{remark}
Of course, once $(\iter \phi n)^*(D)$ contains $N+2$ distinct geometrically-irreducible components, a set $Z$ of $(S, (\iter \phi n)^*(D))$-integral points in all of $\pp^N$ is Zariski-non-dense.  Hence, the set of orbit points which are $(S,D)$-integral is contained in $(\iter \phi n)(Z)$, which is also Zariski-non-dense.  Thus, integral-point statements such as Theorem \ref{thm:sec1_fin} are trivial, while the number-of-digits comparison statements such as Proposition \ref{prop:vojtasemi} are less immediate.
\end{remark}

\begin{example}\label{ex:vojtasemi}
Let $\phi = [ X^3 : L\cdot Q : YZ^2]$ on $\pp^2$, where $L$ is a $\qq$-linear form and $Q$ is a geometrically irreducible $\qq$-quadratic form such that neither goes through $[0:0:1]$ or $[0:1:0]$.  This is a morphism.  Let $P\in \pp^2(\qq)$ be such that the $Y$-coordinate is much larger than the other two coordinates.  $(\iter \phi 2)^*(Z=0) = (L\cdot Q \cdot Y^2 Z^4)$, so $D_1 = (L \cdot Q\cdot  YZ)$ has four distinct components.  Then $D_2 = (Y) + 3(Z)$, and as $L\cdot Q$ contains a nonzero $Y^3$-term while the other two coordinates of $\phi$ do not, the $Y$-coordinate is always the largest in $\mathcal O_\phi(P)$.  Therefore, $\lambda_{v_\infty}((Y), \iter \phi m (P)) = 0$, and hence we have
\[
\lambda_{v_\infty}(D_2, \iter \phi m (P)) \le 3 h(\iter \phi m (P)).
\]
Therefore, the proposition \textit{unconditionally} tells us that
\[
\left\{\iter \phi m(P): \frac{\log |\itersub a2m|}{\log \max (|\itersub a0m|, |\itersub a1m|, |\itersub a2m|)}  \le \frac 19-\epsilon \right\}
\]
is contained in a finite union of algebraic curves.  Note that since $(\iter \phi 2)^*(Z=0)$ does not contain four distinct lines, Proposition \ref{prop:schmidt} does not apply.
\end{example}

We next discuss one other situation where our results become unconditional.  Sometimes, Vojta's inequality can be verified even for non-normal-crossings divisors, and our next example takes advantage of this.

\begin{example}\label{ex:archS}
Let $k=\qq$, $S = \{\infty\}$, and $D_1 = (XYZ(Y+Z) = 0)$ on $\pp^2$.  Since three lines of $D$ go through $[1:0:0]$, $D$ is not normal-crossings.  On the other hand, using \eqref{eq:localht} and writing $P = [a:b:c]$ with integers with gcd $1$, the LHS of \eqref{eq:vojtacon} becomes
\begin{align*}
\lambda_{v_\infty}(D_1,[a:b:c]) - 3 h(P) &= \log \frac{\max(|a|,|b|,|c|)^4}{|a| |b| |c| |b+c|} - 3 \log \max(|a|,|b|,|c|) \\
&= \log \frac{\max(|a|,|b|,|c|)}{|a| |b| |c| |b+c|} \le 0,
\end{align*}
because whichever coordinate has the maximum absolute value, it is canceled by the denominator.  Hence, Vojta's inequality is trivially satisfied for $D_1$, and so if $(\iter \phi n)^*D$ contains $D_1$, one can replace $\itnc Dn$ with $D_1$ and obtain Theorem \ref{thm:main} and its consequences \textit{unconditionally}.  For example, let $\phi = [X^3+2 Y^3 + 3 Z^3 : X^2 Y + YZ^2 + Z^3 : X^2 Y - YZ^2 - Z^3]$ and let $D = ((Y+Z)(Y-Z)=0)$.  Since the common intersections of the last two coordinates of $\phi$ are $[0:1:0], [0:1:-1], [1:0:0]$ and the first coordinate of $\phi$ is nonzero at these points, $\phi$ is a morphism.  Now, $\phi^*D = (4 X^2 Y Z^2(Y+Z) = 0)$, which contains $D_1$ as a subdivisor.  Thus, we conclude unconditionally from Theorem \ref{thm:main} that
\[
\left\{\iter \phi m(P): \frac{\displaystyle \sum_{v\notin S} \lambda_v(D, \iter \phi m(P))}{2 h(\iter \phi m(P))} \le \frac 16 - \epsilon \right\}
\]
is contained in a finite union of algebraic curves for any $\epsilon>0$ and $P\in \pp^2(\qq)$.
\end{example}

We close this section with an example which demonstrates that our theorems are not yet satisfactory for determining Zariski-non-density of integral points in orbits.

\begin{example}\label{ex:bad}
Let $\phi = [X^3 : Y^3: Z^2 (Y-Z)]$.  This is a morphism, and since the last two coordinates of $\phi$ only involve $Y$ and $Z$, this property continues to hold for all $\iter \phi n$. On the other hand, any homogeneous polynomial in $Y$ and $Z$ factors into linear terms, all of which go through the point $[1:0:0]$.  Thus, for $D = (Z=0)$, $\itnc Dn$ is at most two lines for every $n$, so this map does not satisfy the hypothesis of Proposition \ref{prop:schmidt}.

However, we can actually show that an orbit of $P = [a:b:1]$ with $a,b\in \zz$ and $a>b>1$ contains only finitely many $(S,D)$-integral points. To see this, let us write $\iter \phi m (P) = [a_m: b_m :c_m]$ with $a_m,b_m,c_m\in \zz$ with gcd $1$.  Adding primes dividing $b$ to $S$, we see that $b_m$ is always an $S$-unit.  Hence, if $\iter \phi {m+1}(P)$ is integral, then $c_m$ and $b_m-c_m$ are integers which divide $b_m^3$, so they are also $S$-units.  Looking at $\frac{b_m-c_m}{b_m} + \frac{c_m}{b_m} = 1$, we see that having infinitely many $(S,D)$-integral orbit points contradicts the $S$-unit equation.

Note that in this particular example,
\[
\frac{\log c_m}{\log \max(a_m,b_m,c_m)}
\]
does not go to $1$ in general.  Since the last two coordinates of $\phi$ only involve $Y$ and $Z$ and the last coordinate of $\iter \phi 2$ is not a pure power of $Z$, it follows from Silverman's result on $\pp^1$ that $\lim \frac{\log \max(b_m,c_m)}{\log c_m} = \lim \frac{\log b_m}{\log c_m} = 1$.  If we further assume that $b$ is a prime, then one can show by induction that $c_m$ is not divisible by $b$, so there is no cancelation when we compute the next iterate.  Therefore,
\[
\lim_{m\to\infty} \frac{\log c_m}{\log \max(a_m,b_m,c_m)} = \lim_{m\to\infty} \frac{\log c_m}{\log a_m} = \lim_{m\to\infty} \frac{\log b_m}{\log a_m} = \frac{\log b}{\log a}.
\]
Thus, the asymptotic ratio of the number of digits of the coordinates of iterates actually depends on the initial point, rather than controlled by the geometry of the divisor and $\phi$.
In this example, $\phi$ has  a totally ramified fixed point at $[1:0:0]$, so in some sense this is an ``exceptional case,'' just as the polynomials are exceptions in dimension $1$.  On the other hand, finiteness of integral points still holds as above.  We need further research to understand this example in a more theoretical framework.
\end{example}

\section{Generalizations and Further Examples}

In this section, we give two directions to which Theorem \ref{thm:main} can be extended, and give several illuminating examples.  Theorem \ref{thm:dratio}  deals with a certain class of \textit{rational} maps, and Theorem \ref{thm:twomaps} deals with a case when the arithmetic of orbits under one map can be controlled by the geometry of another map.

When there exists $n$ such that $\deg \itnc Dn > N+1$, we know (at least under assuming Vojta's conjecture) from Theorem \ref{thm:main} that the set of orbit points which are $(S,D)$-integral is Zariski-non-dense.  On the other hand, when an iterate of $\phi$ is a polynomial, i.e.  there is a hyperplane $H$ such that $(\iter \phi k)^* (H) = d^k H$, one can easily get Zariski-dense $H$-integral points in an orbit.  In this case, the set $\displaystyle \bigcup_{i=0}^{k-1} (\iter \phi i)^{-1}(H)$ is totally invariant under $\phi$, and so it is conjectured to be a union of hyperplanes (the proof for $\pp^2$ in \cite{shishi} is valid).  Moreover, since $(\iter \phi {k-i})(H)$ is irreducible, $(\iter \phi i)^{-1}(H)$ can only have one irreducible component.  Therefore, $\deg \itnc Hn = 1$ for all $n$ when an iterate of $\phi$ is a polynomial.

It is now natural to ask about Zariski-density of integral points in orbits when $2 \le \sup_n \deg \itnc Dn \le N+1$.  Characterizing such maps is an important first step.  For example, in dimension $1$, if iterates are never polynomials, Riemann-Hurwitz tells us that there is an iterate that has three distinct poles.  Therefore, $\sup_n \deg \itnc Dn$ is equal to $1$ when $\iter \phi 2$ is a polynomial and is at least $3$ otherwise.  So no map satisfies $\sup_n \deg \itnc Dn = 2$ in dimension $1$.  In dimension $2$, $\sup_n \deg \itnc Dn$ is equal to $2$ for maps of similar form to Example \ref{ex:bad}, and these can be characterized by totally ramified fixed points.  It seems difficult to create a map on $\pp^N$ such that $\sup_n \deg \itnc Dn$ is exactly $N+1$, and in fact it may actually be non-existent, just as in the $N=1$ case.  Once we can appropriately characterize maps with $2 \le \sup_n \deg \itnc Dn \le N+1$, one can hope to analyze Zariski-density of integral points in orbits.

If we drop the hypothesis that $\phi$ is a \textit{morphism}, we have the following map satisfying $2 \le \sup_n \deg \itnc Dn \le N+1$ such that the set of integral points in an orbit is Zariski-dense.

\begin{example}\label{ex:ratl_infty_int}
Let $\phi$ be  $[X_0^3: X_1^3: X_1 X_2^2 : X_2 X_3^2 : \cdots :X_{N-1} X_N^2]$ on $\pp^N$.  This is a rational map, undefined for example at $[0:0:\cdots : 0:1]$ but the indeterminancy locus is contained in the hyperplane $X_0 = 0$.  Let $H = (X_N = 0)$.  Then the support of $\phi^*(H)$ is defined by $X_{N-1} X_N = 0$, the support of $(\iter \phi 2)^*(H)$ is defined by $X_{N-2}X_{N-1} X_N =0$, and continuing in a similar way, one sees that the support of $(\iter \phi k)^*(H)$ for $k\ge N-1$ is defined by $X_1 \cdots X_N = 0$.  Thus, $\sup_n \deg \itnc Dn = N$ for this map.

If we let $P = [2^N : 2^{N-1}: \cdots : 2: 1]$, then we immediately see that $\iter \phi m (P)$ is always integral with respect to $H$.  If we let $\itersub aim$ denote the power of $2$ of the $i$-th coordinate of $\iter \phi m(P)$ in the reduced form, all the coordinates of $\iter \phi {m+1}(P)$ are divisible by $2^{\itersub a{N-1}m}$ and it is easy to prove by induction that $\itersub a0m > \itersub a1m > \cdots >  \itersub a{N-1}m$.  Hence, $\itersub a0{m+1} = 3 \itersub a0m - \itersub a{N-1}m > 2 \itersub a0m$, so all orbit points are distinct.  Thus, the orbit contains infinitely many $H$-integral points.

With a bit more work, we can also show that the orbit is Zariski-dense.  Since each coordinate of $\phi$ is a monomial, one can view it as a map on $\ggg_m^{N}$.  Since the orbit points are integral, if they are not Zariski-dense, the Mordell--Lang conjecture on semiabelian varieties \cite{vojta_semiabel} tells us that they are contained in a finite union of translates of subtori.  Thus, for infinitely many $m$'s, $(\itersub a0m, \ldots, \itersub a{N-1}m)$ satisfy a (possibly non-homogeneous) linear equation $c_0 Y_0 + \cdots c_{N-1} Y_{N-1} = c$.  Now, from above discussion, we have the following relation:
\[
\begin{pmatrix}
\itersub a0{m+1}\\\vdots \\ \itersub a{N-1}{m+1}
\end{pmatrix} = \begin{pmatrix}
3  & 0  & \cdots   & \cdots  &  \cdots    & 0 & -1\\0 & 3 & 0 & \cdots    & \cdots  & 0 & -1\\0 & 1 & 2 &  0  & \cdots  & 0 & -1\\0 & 0 & 1 & 2 & 0 & \cdots   & -1\\ \vdots  & \ddots  & \ddots  & \ddots  & \ddots  & \vdots  &  \vdots \\ 0 & \cdots  & \cdots  & \cdots  & 1 & 2 & -1\\ 0 & \cdots   & \cdots  & \cdots    & 0 & 1 & 1
\end{pmatrix} \begin{pmatrix}
\itersub a0m \\ \vdots \\ \itersub a{N-1}m
\end{pmatrix}.
\]
Let $A$ denote the $N\times N$ matrix above.  The Jordan decomposition of $A$ is one $1\times 1$ block of eigenvalue $3$ and one $(N-1)\times (N-1)$ block of eigenvalue $2$, and the transformation matrix is $\begin{pmatrix}
1 & 0 & 0  & \cdots  & 0  & 0 \\ 0  & 1  & 1 & \cdots  & 1  &  1 \\0 & 1 & 1  & \cdots  & 1 & 0\\ \vdots  &  \vdots  & \vdots  & \iddots  & 0 & \vdots \\0 & 1 & 1 & \iddots & \vdots  & \vdots \\ 0  & 1 & 0  & \cdots  & 0 & 0
\end{pmatrix}$, where the first column is an eigenvector of eigenvalue $3$ and the second column of $2$.   From this, it is easy to compute that the first row of $A^m$ has entries in $3^m\cdot \qq[m]$ and the other rows have entries in $2^m \cdot \qq[m]$. So the coefficient of $Y_0$ is zero.  Now, because of the form of the transformation matrix, for any $i\ge 3$, $(i,j)$-entry of $A^m$ agrees with $(2,j)$-entry for $2^m m^{i-2}, \ldots, 2^m m^{N-2}$ terms, so $\itersub a{i-1}m$ agrees with $\itersub a1m$ for $2^m m^{i-2}$ and higher-degree terms.  Hence, $\itersub a1m, \ldots, \itersub a{N-1}m$ are clearly linearly independent.  On the other hand, $2^m (\alpha_0 + \alpha_1 m+\cdots + \alpha_{N-2} m^{N-2}) = \alpha$ can only have finitely many solutions in $m$, since by clearing the denominators of $\alpha_i$'s and $\alpha$, the LHS is at least $2^m$ in absolute value if the polynomial part is nonzero.  This concludes the proof that all the orbit points of $P$ are distinct and $H$-integral, and that they form a Zariski-dense set.
\end{example}

Note that this example is not covered by Theorem \ref{thm:main}, as it is only a rational map.  We will quickly discuss a rare rational map whose arithmetic of orbits can be analyzed by Theorem \ref{thm:main} and its consequences; we will then present a theory which treats more general rational maps, including Example \ref{ex:ratl_infty_int}.

\begin{example}\label{ex:ratl_wo_dratio}
Let $\phi = [Y^5: \left(\prod_{i=1}^4 L_i\right) Z : Z^5]$ on $\pp^2$ defined over $\qq$, where $L_i(X,Y,Z) = \alpha_i X + \beta_i Y + \gamma_i Z$ are lines in general position with $\alpha_i \neq 0$.  This is undefined at $[1:0:0]$, and
\[
\iter \phi 2 = \left[Z^5 \prod_{i=1}^4 L_i(X,Y,Z)^5  : Z^5 \prod_{i=1}^4 \left(\alpha_i Y^5 + \beta_i Z \prod_{j=1}^4 L_j(X,Y,Z) + \gamma_i Z^5\right) : Z^{25}\right].
\]
Upon canceling the common $Z^5$, we see that $\iter \phi 2$ is actually a \textit{morphism} of degree $20$, since plugging in $Z=0$ into the second coordinate results in $Y=0$, and plugging these into the first coordinate leads to $X=0$.  For $D = (X=0)$, $(\iter \phi 2)^*D$ contains $L_i$'s, so applying Theorem \ref{thm:main} (actually Proposition \ref{prop:schmidt} and Corollary \ref{cor:coords}) to the map $\iter \phi 2$ twice, first to the orbit of $P\in \pp^2(\qq)$ and second to the orbit of $\phi(P)$, we unconditionally conclude that
\[
\left\{\iter \phi m(P): \frac{\log |\itersub a0m|_S'}{\log \max (|\itersub a0m|, |\itersub a1m|, |\itersub a2m|)}  \le \frac 1{20}-\epsilon \right\}
\]
is contained in a finite union of algebraic curves.
\end{example}

Of course, a rational map usually stays a rational map upon composition, and in general (arithmetic) dynamics for rational maps is much more difficult than for morphisms.  To obtain Theorem \ref{thm:main}-like results for rational maps, we note that the reason for the morphism assumption in Theorem \ref{thm:main} is that the height inequality \eqref{lowerht} does not hold for rational maps.  On the other hand, once something similar to \eqref{lowerht} is obtained, the proof should go through.  We remark that Lee \cite{joey_aa} has recently proved a height inequality for rational maps whose indeterminancy is contained in a hyperplane.  We recall his definition of ``$D$-ratio,'' comparing the coefficients of irreducible divisors in two divisors: one is the pullback of the hyperplane via the blowup map, and the other is the pullback of the hyperplane via the morphism that resolves the indeterminancy.

\begin{definition}
Let $\phi:\pp^N\dashrightarrow \pp^N$ be a rational map whose indeterminancy locus is contained in a hyperplane $H$.  Let $\pi: V\too \pp^N$ be a sequence of monoidal transformations (with smooth centers) which resolves the indeterminancy, with $\widetilde\phi: V\too \pp^N$ the resolved morphism.  Then the $D$\textit{-ratio} $r(\phi)$ of $\phi$ is defined to be the minimum $r$ such that the $\qq$-divisor $\frac r{\deg \phi} (\widetilde \phi)^*H - \pi^* H$ is linearly equivalent to a divisor which is a nonnegative linear combination of the strict transform of $H$ and the exceptional divisors from $\pi$. The D-ratio depends on the choices of $H$ and on the blowups that resolve $\phi$.
\end{definition}

Lee \cite[Proposition 4.5]{joey_aa} shows that $r(\phi)\ge 1$ and equal to $1$ if $\phi$ is a morphism.  Further, he shows the following height inequality \cite[Theorem A]{joey_aa}:
\begin{equation}\label{eq:dratio_ht}
\frac r{\deg \phi} h(\phi(P)) > h(P)-C \qquad \text{ for }P\notin H.
\end{equation}
We will now use this to obtain an analog of Theorem \ref{thm:main} for rational maps.

\begin{theorem}\label{thm:dratio}
Let $\phi:\pp^N\dashrightarrow \pp^N$ be a dominant rational map defined over $\overline \qq$ whose indeterminancy is contained in a hyperplane $H$ and whose degree is $d\ge 2$.  Let $r$ be its D-ratio.  Let $D$ be a divisor on $\pp^N$ defined over $\overline \qq$, and let $\itnc Dn$  the normal-crossings part of $(\iter \phi n)^*(D)$.  Let $e_n = \deg (\iter \phi n)$.  Let $k$ be a number field that contains the fields of definition of $\phi$, $D$, and $\itnc Dn$.  Let us assume Vojta's conjecture for the divisor $\itnc Dn$.  If
\begin{equation}\label{eq:dratio_cond}
e_n - \frac{\deg \itnc Dn -  (N+1)}{\deg(D)} < \left(\frac dr\right)^n,
\end{equation}
then for any $P\in \pp^N(k)$ with $\mathcal O_\phi(P) \cap H = \emptyset$ and for any finite set $S\subset M_k$, the set $\mathcal O_\phi(P) \cap (\pp^N\backslash D)(R_S)$ is Zariski-non-dense. More precisely, if we let $c_n =1 - \left(e_n - \frac{\deg \itnc Dn - (N+1)}{\deg(D)}\right)\left(\frac rd\right)^n$, then
\[
\left\{\iter\phi m(P): \,\,\frac{\displaystyle \sum_{v\notin S} \lambda_v(D, \iter \phi m(P))}{\deg (D) h(\iter\phi m (P))} \le c_n - \epsilon\right\}
\]
is Zariski-non-dense for any $\epsilon>0$.
\end{theorem}

\begin{remark} We make several observations.
\begin{itemize}
\item[(i)] For morphisms, $r=1$ and $e_n = d^n$, so the condition for integral points in orbits to be Zariski-non-dense agrees with Theorem \ref{thm:sec1_fin} and the $c_n$ agrees with that in Theorem \ref{thm:main}.  As $r$ gets bigger, one generally needs more degrees of the normal-crossings part to conclude Zariski-non-denseness of integral points in orbits.

\item[(ii)] Although \cite{joey_aa} assumes that maps are polynomials for all of the dynamical results, his height inequality \eqref{eq:dratio_ht} holds away from the hyperplane, so for our purposes it is enough to assume that the forward orbit does not intersect the hyperplane.

\item[(iii)] By applying this theorem with $\psi = \iter \phi n$ and observing that the orbit of $P$ under $\phi$ is a finite union of orbits under $\psi$ (cf. the end of Example \ref{ex:ratl_wo_dratio}), we see that \eqref{eq:dratio_cond} can be replaced by
\[
e_n - \frac{\deg \itnc Dn -  (N+1)}{\deg(D)} < \max\left(\left(\frac dr\right)^n, \,\,\,\frac{e_n}{r(\iter \phi n)}\right).
\]

\item[(iv)] By \cite[Proposition 4.5]{joey_aa},
it follows that $e_n\ge (d/r)^n$.  If $\phi$ is a polynomial on $\aa^N = \pp^N\backslash H$, then $\itnc Hn = H$ for all $n$, so we see that \eqref{eq:dratio_cond} will always fail.  This is consistent with the fact that polynomial mappings can have infinitely many integral points in orbits.
\end{itemize}
\end{remark}

\begin{proof}
The proof is very similar to that of Theorem \ref{thm:main}, but we need to accommodate for the fact that $\phi$ is only a rational map. Specifically, the degree of $\iter \phi n$ might not be $d^n$, and the height inequality \eqref{lowerht} needs to be replaced by the one with the D-ratio \eqref{eq:dratio_ht}.  Note that $\deg((\iter \phi n)^*D) = e_n \deg(D)$.  Vojta's conjecture for $\itnc Dn$ tells us that there exists a constant $C$ and a finite union $Z_\epsilon$ of hypersurfaces such that
\begin{equation*}
\sum_{v\in S} \lambda_v \left(\itnc Dn, Q\right) < \left(N+1+\frac {\epsilon d^n\deg(D)}{r^n}\right) h(Q) + C
\end{equation*}
holds for all $Q\in \pp^N(k)\backslash Z_\epsilon$.  We also have
\begin{equation*}
\sum_{v\in S} \lambda_v \left(\iter Dn - \itnc D n, Q\right) \le h_{\iter Dn - \itnc D n}(Q) +C_1 \le \left(e_n \deg (D) - \deg \itnc Dn\right) h(Q) + C_1'.
\end{equation*}
Thus, for $Q\notin Z_\epsilon$ such that $\mathcal O_\phi(Q)\cap H = \emptyset$,
\begin{align*}
\sum_{v\in S} &\lambda_v(D, \iter \phi n(Q)) \le \sum_{v\in S} \lambda_v((\iter \phi n)^*(D), Q) + C_2 \\
&<\left( e_n \deg(D) - \deg \itnc Dn + N+1 + \frac {\epsilon d^n \deg(D)}{r^n} \right) h(Q) + C_3 \\
&\le \left(e_n  - \frac{\deg \itnc Dn - (N+1)}{\deg(D)}\right)\left(\frac rd\right)^n  \deg(D) h(\iter \phi n (Q))\\
 & \phantom{le \left(e_n  - \frac{\deg \itnc Dn - (N+1)}{\deg(D)}\right)\left(\frac rd\right)^n}+ \epsilon \deg(D) h(\iter\phi n(Q)) + C_4 \quad \text{ by \eqref{eq:dratio_ht}}\\
&\le \left(e_n  - \frac{\deg \itnc Dn - (N+1)}{\deg(D)}\right)\left(\frac rd\right)^n  h_D(\iter \phi n (Q)) + \epsilon  h_D(\iter\phi n(Q)) + C_5.
\end{align*}
For any $m\ge n$, letting $Q = \iter \phi {m-n}(P)$, we have
\[
\sum_{v\notin S} \lambda_v(D, \iter \phi m(P)) > \left(c_n - \epsilon \right) h_D(\iter \phi m(P)) - C_6,
\]
as long as $\iter \phi {m-n}(P) \notin Z_\epsilon$ and $\mathcal O_\phi(P) \cap H= \emptyset$.  The rest of the argument is the same as Theorem \ref{thm:main}.
\end{proof}

\begin{example}\label{ex:dratio}
This is an example of Theorem \ref{thm:dratio} for which Vojta's conjecture is not necessary.  Let $M_i$ be the linear form $X_0 + X_1 + X_2 + X_3 - X_i$ and let $\phi$ be the self-map
\[
\left[M_0 \prod_{j=1}^{d-1} L_{0,j} \,\,:\,\, M_1 \prod_{j=1}^{d-1} L_{1,j} \,\,:\,\, M_2 \prod_{j=1}^{d-1} L_{2,j} \,\,:\,\, X_0^{d-2} X_1 X_2 \,\,:\,\, M_3 \prod_{j=1}^{d-1} L_{3,j}\right]
\]
on $\pp^4$ (note that the last coordinate is indexed with $3$'s for convenience), where $L_{i,j}$'s are linear forms satisfying the following properties:

\begin{itemize}
\item[(1)] \label{dratio_cond1} $L_{i,j}$'s have  nonzero coefficients for $X_4$, and $\{L_{0,1}, \ldots, L_{2,d-1}, M_0, M_1, M_2\}$ is in general position (i.e. any five of them are linearly independent).

\item[(2)] \label{dratio_cond2} $\{P: L_{0,j_0}(P) = 0\}\cap \cdots\cap \{P: L_{3,j_3}(P) = 0\}$ is empty on the hyperplanes $X_i = 0$ in $\pp^4$ for $i=0,\ldots, 2$.

\item[(3)] \label{dratio_cond5} For any $i\in \{0,1,2\}$ and $k\in \{0,\ldots, 3\}$, the $\pp^2$ defined by $X_i = M_k = 0$ does not contain any point of $\bigcap_{\ell \neq k} \{P: L_{\ell, j_\ell}(P) = 0\}$.

\item[(4)] \label{dratio_cond8} For any $i\in \{0,1,2\}$ and $k, m\in \{0,\ldots, 3\}$ with $k\neq m$, the $\pp^1$ defined by $X_i = M_k = M_m = 0$ does not contain any point of $\bigcap_{\ell \neq k, m} \{P: L_{\ell, j_\ell}(P) = 0\}$.

\end{itemize}

We will have an explicit example of the $L_{i,j}$'s at the end of this example, but it is intuitively easy to see that linear forms in general enough position will satisfy the above.  Indeed, for (2), restricting linear forms to $X_i = 0$ makes them into hyperplanes in $\pp^3$, and four hyperplanes in $\pp^3$ in general position do not meet (over $\overline \qq$).   For (3), each linear form becomes a line on $\pp^2$ upon restriction, three of which in general position do not meet.  Finally, for (4), a linear form becomes a point on  $\pp^1$ upon restriction, and two general points do not coincide.  Thus, linear forms in general position should work.

Now, let us see that the conditions above \textit{unconditionally} ensure Zariski-non-denseness of integral points in orbits.  First, we will show that the only indeterminancy of $\phi$ is $[0:0:0:0:1]$.  If $P$ is an indeterminancy point, from the second to last coordinate, $P$ is on the hyperplane $X_i = 0$ for some $i=0, 1, 2$.  Let us first suppose that none of the $M_k$'s vanishes at $P$.  Then one of the linear forms from each coordinate of $\phi$ must vanish at $P$, contradicting (2).  Next, suppose that exactly one $M_k$ vanishes at $P$. Then for each $\ell\in \{0,\ldots, 3\}$ not equal to $k$, one of the linear forms $L_{\ell,j}$ must vanish at $P$, contradicting (3).  Thirdly, suppose that exactly two of the $M_k$'s vanish at $P$, say $M_k$ and $M_m$.  Then for each $\ell\neq k,m$, one of the linear forms $L_{\ell, j}$ must vanish at $P$, contradicting (4).  Finally, any three of the $M_k$'s and $X_i$ are linearly independent, so their simultaneous vanishing at $P$ would imply that $X_0 = X_1 = X_2 = X_3 = 0$, namely $P = [0:0:0:0:1]$.

Now we blow up $P$.  The key computation turns out to be completely symmetric with $X_0,\ldots, X_3$, so let us work on a patch where $X_0= x_0, X_1 = x_0 x_1', X_2 = x_0 x_2', X_3 = x_0 x_3', X_4 = 1$.  Since the blowup does not affect $\phi$ outside the center, we only need to look at the exceptional divisor, defined by $x_0 = 0$ on this patch, to see if this blowup resolves the indeterminancy.  Here, $\phi$ is defined by
\begin{align*}
&\left[x_0 (x_1'+x_2'+x_3') \prod L_{0,j}(0,0,0,0, 1): x_0 (1+x_2'+x_3') \prod L_{1,j}(0,0,0,0,1) \right. \\
&\phantom{abcabca}\left. : x_0 (1+x_1'+x_3') \prod L_{2,j}(0,0,0,0,1) : x_0 x_0^{d-1} x_1' x_2' \right.\\
&\phantom{abcabca}\phantom{abcabca}\phantom{abcabca}\phantom{abcabca}\left. : x_0 (1+x_1'+x_2') \prod L_{3,j}(0,0,0,0,1)\right],
\end{align*}
so one factor of $x_0$ comes out.  By  condition (1), all the linear forms evaluated at $(0,0,0,0,1)$ are nonzero.  Thus, if this map were  undefined, from the second, third, and the fifth coordinates, we must have $x_2' + x_3' = x_1' +x_3' = x_1' + x_2' = -1$.  But then $x_1' + x_2' + x_3' = -3/2$, and the first coordinate will be nonzero.  Therefore, this map is well-defined everywhere on the exceptional divisor, and hence the indeterminancy of $\phi$ is resolved by this one blowup.  Defining $H$ to be the hyperplane $X_3 = 0$, we see from the calculation above that $\widetilde \phi^*(H)$ has $(d-1) E$.  Thus, $\pi^* H = \widetilde H + E$ and $\widetilde \phi^*H = d \widetilde H + (d-1) E$, so the D-ratio is
\[
d\cdot \max\left(\frac 1d, \,\, \frac 1{d-1}\right) = \frac d{d-1}.
\]
On the other hand, $\left(\iter \phi 2\right)^*H$ is defined by $M_0^{d-2} M_1 M_2 \prod_{j} L_{0,j}^{d-2} \prod_j L_{1,j} \prod_j L_{2,j}$, so by (1) this divisor contains $3d$ linear forms in general position.  Thus, \eqref{eq:dratio_cond} for the second iteration becomes
\[
d^2 - 3 d + 5 < (d-1)^2.
\]
This is satisfied for $d\ge 5$, so Theorem \ref{thm:dratio} and Proposition \ref{prop:schmidt} apply.  We \textit{unconditionally} conclude  that if $\iter \phi m(P) = [\itersub a0m:\cdots :\itersub a4m]$ with $\itersub a3m \neq 0$ for all $m$, then
\[
\left\{\iter \phi m(P): \,\, \frac{\log |\itersub a3m|_S'}{\log \max_i |\itersub aim|} \le \frac{d-4}{(d-1)^2}-\epsilon \right\}
\]
is Zariski-non-dense for any $\epsilon>0$.

As for a specific example of the $L_{i,j}$'s, let $d=5$ and let 
\[
L_{i,j} = X_0 + (4i+2+j) X_1 + (4i+2+j)^2 X_2 + (4i+2+j)^3 X_3 + (4i+2+j)^4 X_4. 
\]
For (1), the independence of five of the $L_{i,j}$'s comes from the Vandermonde determinant.  By a tiresome computation, we can also confirm that one or more of $M_0, M_1, M_2$ together with $L_{i,j}$'s are in general position as well, confirming (1) (if we want to avoid computation, we can just use the $L_{i,j}$'s, resulting in $\deg \itlin D2 = 3d-3$ and needing to assume $d\ge 8$ instead).  For (2), we check that a Vandermonde determinant with one of the columns removed is nonzero.  For the last two conditions, let us assume that $i=0$, i.e. $X_0 = 0$; other cases are similar.  For (3) with $k=0$, we have
$X_0=0$ and $X_1 = -X_2 -X_3$, so we need to check that the determinant
\[
\begin{vmatrix}
a^2-a  & a^3 -a  & a^4\\b^2-b & b^3 -b & b^4\\c^2-c & c^3-c & c^4
\end{vmatrix}
\]
is nonzero.  This determinant is $- abc (a-b)(a-c)(b-c)[(a-1)(b-1)(c-1)+1]$, 
so this is clearly nonzero for distinct integers $\ge 3$.  Similarly, for $k=1$ (resp. $k=2, 3$), we check that the $3\times 3$ determinant where each row has the shape $a$, $a^3 - a^2$, $a^4$ (resp. $a^2, a^3-a, a^4$, and $a^2-a, a^3, a^4$) is nonzero for distinct integers $\ge 3$.

For (4), when $(k,m) = (0,1)$, we have $X_0=X_1=X_2+X_3 = 0$.  Since the function $(n^3-n^2)/n^4$ is a strictly decreasing function on the integers $\ge 3$, the values at distinct $a$ and $b$ will be different.  Similarly, for $(k,m) = (0,2), (0,3),(1,2),(1,3)$, $(2,3)$, the corresponding functions
\[
\frac{n^3-n}{n^4},\quad \frac{n^2-n}{n^4}, \quad \frac{n^3-n^2-n}{n^4}, \quad \frac{n^3 - n^2 + n}{n^4}, \quad \frac{n^3+n^2-n}{n^4}
\]
are all strictly decreasing functions on integers $\ge 3$, confirming (4).
This finishes the verification that $L_{i,j}$'s satisfy conditions (1)-(4).
\end{example}

\begin{remark}
While a similar construction works for dimension $> 4$, the above construction does not quite work with dimensions $2$ or $3$.  For dimension $2$, we only have 2 coordinates where we can choose linear forms, so the LHS of \eqref{eq:dratio_cond} will be at best $d^2 - 2 d + 3$, which is bigger than $(d-1)^2$ for all $d$.  On $\pp^3$, since the indeterminancy $[0:0:0:1]$ has to be on the hyperplane $H$ used to define the $D$-ratio, the term that produces many exceptional divisors in $(\widetilde \phi)^*H$  (such as $X_0^{d-2} X_1 X_2$) cannot be the last coordinate of $\phi$.  This again forces us to have just two coordinates where we can choose linear forms.  Choosing $D$ to be $H$ as in this example seems to be efficient, but of course unnecessary.  It will be interesting to come up with other constructions to have unconditional examples of Theorem \ref{thm:dratio} on $\pp^2$ and $\pp^3$.
\end{remark}

Finally, we discuss one more extension.  This stems from the following observation: the fact that the points are in an orbit was only minimally used in all of the arguments above.  What \textit{is} crucially being used is $\iter \phi m(P)$ has an inverse image by the map $\iter \phi n$, namely $\iter \phi {m-n}(P)$, which is also defined over the field for which $P$ and $\phi$ are defined.  Since Vojta's Conjecture (Conjecture \ref{con:vojta}) is a height inequality over a fixed number field, it is important that we do not need a field extension for each $m$.  The next theorem exploits this observation.
\begin{theorem}\label{thm:twomaps}
Let $\phi: \pp^N \too \pp^N$ be a morphism defined over $\overline \qq$ of degree at least $2$, and let $P\in \pp^N(\overline \qq)$. Let $\psi: \pp^N\too \pp^N$ be another morphism over $\overline \qq$ of degree $d \ge 2$, and let $D$ be a divisor on $\pp^N$ with $\itncpsi Dn$  the normal-crossings part of $(\iter \psi n)^*(D)$.  Assume that $\deg \itncpsi Dn > N+1$ for some $n$, and further assume the ``inverse image property'': there exists a number field $k$ such that for each $m$, there exists $Q_m\in \pp^N(k)$ such that $\iter \phi m (P) = \iter \psi n(Q_m)$.  Then letting $c_n = \frac{\deg \itncpsi Dn - (N+1)}{d^n\deg(D)}$ and assuming Vojta's conjecture for the divisor $\itncpsi Dn$, for all finite subset $S$ of $M_k$ and $\epsilon >0$, the set
\[
\left\{\iter \phi m(P): \frac{\displaystyle \sum_{v\notin S} \lambda_v(D, \iter \phi m(P))}{\deg(D) h(\iter \phi m(P))}  \le c_n-\epsilon  \right\}
\]
is Zariski-non-dense.
\end{theorem}

\begin{remark}
Note that the choice of $\psi$ can depend on $P$. This theorem is the analog of Theorem \ref{thm:main}, and analogs of Corollaries \ref{cor:sup}, \ref{cor:generic}, and \ref{cor:coords} also hold for this situation.
\end{remark}

\begin{proof}
Let $D_\psi^{(n)} = (\iter \psi n)^*(D)$.   The argument is essentially the same as the proof of Theorem \ref{thm:main}, replacing $\phi$ by $\psi$ when appropriate.  Enlarging $k$ if necessary, we may assume that $\itncpsi Dn$ and $\psi$ and $\phi$ are defined over $k$.  By Vojta's conjecture for the divisor $\itncpsi Dn$ over $k$, there exist constants $C$ and $C_1$ and a union $Z_\epsilon$ of hypersurfaces such that
\begin{align*}
\sum_{v\in S} &\lambda_v \left(\itncpsi Dn, Q\right) < (N+1+ \epsilon d^n\deg(D)) h(Q) + C \quad Q\in \pp^N(k) \backslash Z_\epsilon\\
\sum_{v\in S} &\lambda_v \left(D_\psi^{(n)} - \itncpsi D n, Q\right) \le h_{D_\psi^{(n)} - \itncpsi D n}(Q)+C_1 \\
&\phantom{\lambda_v \left(D_\psi^{(n)} - \itncpsi D n, Q\right)}\le (d^n \deg(D) - \deg \itncpsi Dn) h(Q) + C_1'.
\end{align*}
As before, as long as $Q_m\notin Z_\epsilon$, the addition of the two inequalities gives us
\begin{align*}
\sum_{v\in S} \lambda_v(D, &\iter \phi m(P)) = \sum_{v\in S} \lambda_v(D, \iter \psi n(Q_m)) = \sum_{v\in S} \lambda_v(D_\psi^{(n)}, Q_m) + C_2 \\
&<\left( d^n \deg(D) - \deg \itncpsi Dn + N+1 + \epsilon d^n \deg(D)\right) h(Q_m) + C_3\\
&\le (1-c_n + \epsilon ) \deg(D) h(\iter \psi n (Q_m)) + C_4 \\
&\le(1-c_n + \epsilon) h_D(\iter \phi m (P)) + C_5.
\end{align*}
The rest of the argument is the same as Theorem \ref{thm:main}.
\end{proof}

\begin{example}
When $\psi$ commutes with $\phi$, the hypothesis of the theorem is satisfied.  Indeed, given a fixed preimage $Q$ of $P$ via $\iter \psi n$, we have $\iter \phi m (P) = \iter \phi m (\iter \psi n(Q)) = \iter \psi n (\iter \phi m(Q))$.  Thus, we can set $Q_m = \iter \phi m (Q)$ and $k$ to be a number field containing the fields of definition of $Q$ and $\phi$. On the other hand, a general morphism does not have another nontrivial morphism that commutes with it just for the dimension reasons, and in fact commuting pairs of endomorphisms on $\pp^N$ have been classified by Dinh--Sibony \cite{dinh_sib}.
\end{example}

We make a final remark that all the results we obtained can be combined, so for example Theorem \ref{thm:twomaps} can work with rational maps using $D$-ratio and Theorem \ref{thm:dratio} becomes unconditional if the relevant divisor satisfies the hypothesis of Proposition \ref{prop:vojtasemi}.

We end the paper with some open questions.  A first obvious one is computing $c = \sup_n \frac{ \deg \itnc Dn - (N+1)}{d^n \deg(D)}$.  The normal-crossings condition is not an easy definition to work with, so in general it is very difficult to compute $c$.  It would be great to have an example of $c=1$ for which all the relevant $\itnc Dn$ are hyperplanes in general position. This can potentially give us an unconditional example of Corollary \ref{cor:generic} (ii).  Secondly, as remarked before Example \ref{ex:ratl_infty_int}, it would be beneficial to geometrically characterize morphisms on $\pp^N$ for which $2\le \sup_n \deg \itnc Dn\le N+1$, as this falls between the polynomial cases and the cases that can be treated by Theorem \ref{thm:main}. Some examples were discussed in \cite{taiwan}, but a more complete characterization should enable us answer whether integral points in orbits under these maps are Zariski-dense. Thirdly, Example \ref{ex:bad} shows that the hypothesis in Theorem \ref{thm:sec1_fin} is not yet a necessary condition. Is there a better geometric criterion, for instance by analyzing whether there are totally ramified fixed subvarieties?  Fourthly, it would be great to generalize arguments in Example \ref{ex:schmidt} to work for more general maps which have enough degrees in $\itlin Dn$.  As remarked there, the methods to go from Zariski-non-dense to finite are standard but adhoc.  This problem is related to the (rank-one) dynamical Mordell--Lang question, so it has an importance on its own.  Finally, it would be illuminating to find $\phi$ and $\psi$ which do not commute but still satisfy the ``inverse image property'' of Theorem \ref{thm:twomaps} for some point $P$.  If the space can be decomposed into two directions in some sense and if the two maps commute in one direction and the point is preperiodic in another direction, it might be possible. These are all interesting questions, meriting further study.

\vspace{0.1in}
\noindent \textit{Acknowledgements.}  I would like to thank Sinnou David, Joseph Silverman, Thomas Tucker and Umberto Zannier for helpful discussions.

\bibliographystyle{amsplain}
\bibliography{nevan}


\end{document}